\documentclass[10pt]{amsart}

\usepackage{amssymb,latexsym}

\usepackage{enumerate}


\newtheorem{thm}{Theorem}[section]

\newtheorem{cor}[thm]{Corollary}
\newtheorem{lem}[thm]{Lemma}

\theoremstyle{definition}

\newtheorem{xrem}{Remark}
\numberwithin{equation}{section}

\newcommand{\R}{{\mathbb{R}}}
\newcommand{\C}{{\mathbb{C}}}

\newcommand{\tRe}{\textup{Re}}

\renewcommand{\a}{\alpha}

\newcommand{\g}{\gamma}
\newcommand{\G}{\Gamma}

\newcommand{\s}{\sigma}

\begin{document}

\baselineskip=17pt

\title[Explicit Bounds for $L$-Functions on the Edge of the Critical Strip ]{Explicit Bounds for $L$-Functions on the Edge of the Critical Strip}

\author[A. Lumley]{Allysa Lumley}

\address{Department of Mathematics and Statistics\\
York University\\
4700 Keele St\\
Toronto, Ontario\\
M3J 1P3 Canada}
\email{alumley@yorku.ca}

\date{}

\begin{abstract}
Assuming GRH and the Ramanujan-Petersson conjecture we prove explicit bounds for $L(1,f)$ for a large class of $L$-functions $L(s,f)$, which includes $L$-functions attached to automorphic cuspidal forms on $GL(n)$.  The proof  generalizes work of Lamzouri, Li and Soundararajan. Furthermore, the main results improve the classical bounds of Littlewood 
\[(1+o(1))\left(\frac{12e^{\gamma}}{\pi^2}\log\log C(f)\right)^{-d} \leq |L(1,f)|\leq  (1+o(1))\Big(2e^{\gamma}\log\log C(f)\Big)^d,\]
where $C(f)$ is the analytic conductor of $L(s,f)$.
\end{abstract}
\subjclass[2010]{Primary 11M06, 11M36; Secondary 11M41.}

\keywords{L-functions, explicit results.}

\maketitle

\section{Introduction}
In analytic number theory, and increasingly in other surprising places, $L$-functions show up as a tool for describing interesting algebraic and geometric phenomena. In particular, understanding the value of $L$-functions on the $1$-line has a number of  applications. For example, the non-vanishing of the Riemann zeta function for $\zeta(1+it)$, $t\in \R$, proves the celebrated Prime Number Theorem.  Additionally, understanding the value $L(1,\chi)$ for certain Dirichlet characters, provides us with insight to the order of the class group  of imaginary quadratic fields through Dirichlet's Class Number Formula.  
Unconditionally, for any non-trivial Dirichlet character $\chi$ with conductor $q$, we have 
\[\frac1{q^{\epsilon}}\ll |L(1,\chi)|\ll \log q.\]
In fact, we can improve the lower bound to $(\log q)^{-1}$, excluding some exceptional cases related to Landau-Siegel zeros (see \cite[Chapter 14]{Davenport}). Louboutin \cite{Louboutin} proves an explicit upper bound of this shape.
Under the assumption of the Generalized Riemann Hypothesis (GRH), we have the much stronger bounds due to Littlewood \cite{Littlewood}:
\[\frac{\zeta(2)(1+o(1))}{2e^{\gamma}\log\log q}\leq |L(1,\chi)|\leq (2e^{\gamma}+o(1))\log\log q,\] 
where $o(1)$ tends to $0$ as $q\to \infty$. Recently, Lamzouri, Li and Soundararajan gave the following explicit refinement
\begin{thm}\cite[Theorem 1.5 ]{LamLiSound}
\label{mainthmlamlisound}
Asume GRH. Let $q$ be a positive integer and $\chi$ be a primitive character modulo $q$. For $q\ge 10^{10}$ we have 
\[|L(1,\chi)|\le 2e^{\gamma}\left(\log\log q-\log2+\frac12+\frac{1}{\log\log q}\right)\]
and 
\[\frac1{|L(1,\chi)|}\le \frac{12e^{\g}}{\pi^2}\left(\log\log q -\log2+\frac12+\frac{1}{\log\log q}+\frac{14\log\log q}{\log q}\right).\]
\end{thm}

The goal of this paper is to provide explicit upper and lower bounds for a large class of $L$-functions, including $L$-functions attached to automorphic cuspidal forms on $GL(n)$. More precisely,  we bound the quantity $|L(1,f)|$, where  $L$  is a degree $d\ge 1$ $L$-function and $f$ is some arithmetic or geometric object. The results will be valid under the assumption of GRH and the Ramanujan-Petersson conjecture.  Additionally, we improve on the bound that comes from generalizing Littlewood's technique, which under both GRH and Ramanujan-Petersson conjecture provides
$$
(1+o(1))\left(\frac{12e^{\gamma}}{\pi^2}\log\log C(f)\right)^{-d} \leq |L(1,f)|\leq  (1+o(1))\Big(2e^{\gamma}\log\log C(f)\Big)^d,
$$
where $o(1)$ is a quantity that tends to $0$ as $C(f)\to\infty$. Here $C(f)$ denotes the analytic conductor of the $L$-function. A precise definition of $C(f)$ along with what the term $L$-function describes will be provided after another example. Other works discussing explicit bounds for higher degree $L$-functions focus on  bounding $L(\tfrac12,f)$, we refer the reader to \cite{Chandee} for details.   \\

We provide a degree $2$ example before appealing to the precise definitions. 
Let $k,q\ge 1$ be integers and let $\chi$ be a Dirichlet character modulo $q$. 
Take $f$ to be a Hecke cusp form of weight $k$, level $q$, 
and character $\chi$, 
with the following Fourier expansion at the cusp $\infty$,
\[f(z)=\sum_{n\ge 1}\lambda_f(n)n^{(k-1)/2}e(nz), \, e(z)=e^{2\pi i z}.\]
 Then
\[L(s,f)=\prod_p\left(1-\frac{\lambda_f(p)}{p^{s}}+\frac{\chi(p)}{p^{2s}}\right)^{-1}=\sum_{n=1}^{\infty}\frac{\lambda_f(n)}{n^s},\]
is a degree $2$ $L$-function. By works of Deligne \cite{Deligne} and Deligne and Serre \cite{DeSer}, it is known that $L(s,f)$ satisfies Ramanujan-Petersson for all weights $k\ge 1$.  In this situation, the analytic conductor is given by $$C(f)=\frac{q}{\pi^2}\left(\frac{1+(k-1)/2}2\right)\left(\frac{1+(k+1)/2}2\right)\asymp qk^2.$$  We deduce the following corollary from our main results Thereom \ref{mainup} and Theorem \ref{mainlow} below.
 \begin{cor}
 Under the assumption of $GRH$, if $\log C(f)\ge 46$, we have 
\[|L(1,f)|\le (2e^{\gamma})^2\left((\log\log C(f))^2-(2\log4-1)\log\log C(f)+(\log4)^2-\log4+2.51\right),\]
and 
\begin{multline*}
\frac1{|L(1,f)|}\le\left(\frac{12e^{\gamma}}{\pi^2}\right)^2\left((\log\log C(f))^2-(2\log4-1)\log\log C(f)+(\log4)^2-\log4+2.67\right. \\ \left. +\frac{89.40((\log\log C(f))^2-2\log4\log\log C(f)+\log^24)}{\log C(f)}\right). 
\end{multline*}
\end{cor}
\subsection{Definitions and Notation}
To begin, let $d\geq 1$ be a fixed positive integer,
and let $L(s, f)$ be given by the Dirichlet series and
Euler product
$$
L(s,f)=\sum_{n=1}^{\infty}\frac{\lambda_f(n)}{n^s}=
 \prod_{p}\prod_{j=1}^d\left(1-\frac{\alpha_{j,f}(p)}{p^{s}}\right)^{-1},
 $$
where $\lambda_f(1)=1$, and  both the series and product are absolutely convergent in $\tRe(s)>1$. We shall assume that $L(s,f)$ satisfies the Ramanujan-Petersson conjecture which states that $|\alpha_{j,f}(p)|\leq 1$ for all primes $p$ and $1\leq j\leq d$. 
Further,
 we define the gamma factor
$$
\gamma(s,f)= \pi^{-ds/2}\prod_{j=1}^d\Gamma\left(\frac{s+\kappa_{j}}{2}\right),
$$
where $\kappa_j$ are complex numbers.  These $\kappa_j$ are called the local parameters at infinity and may be referred to as such throughout. In general, it is assumed that $\tRe(\kappa_j)>-1$, in our case the Ramanujan-Petersson conjecture guarantees that $\tRe(\kappa_j)\geq 0$.
This last condition ensures that $\gamma(s,f)$ has no pole in $\tRe(s)>0.$
Furthermore, there exists a positive integer $q(f)$ (called the conductor of $L(s,f)$),
 such that the completed $L$-function,
$$
\xi(s,f)=q(f)^{s/2}\gamma(s,f)L(s,f),
$$
has an analytic continuation to the entire
complex plane, and has finite order. This completion satisfies a functional
 equation
$$
\xi(s,f)=\epsilon(f)\xi(1-s,\overline{f}),
$$
where $\epsilon(f)$ is a complex number of absolute value $1$, and
$\xi(s,\overline{f})=\overline{\xi(\overline{s},f)}$ ($\overline{f}$ is called the dual of $f$). 
Uniform estimates for analytic quantities associated to $L(s,f)$, when $L(s,f)$ is varying rely on a number of parameters, it is therefore convenient to state the results in terms of the analytic conductor which we define as follows: For $s\in \C$,
$$
C(f,s):= \frac{q(f)}{\pi^d}\prod_{j=1}^d\left|\frac{s+\kappa_j}{2}\right|.
$$
In this article we are interested in studying the value of $L()$
$$
C(f):=C(f,1)= \frac{q(f)}{\pi^d}\prod_{j=1}^d\left|\frac{1+\kappa_j}{2}\right|.
$$
We note that in \cite{Chandee} the author uses $C(f)=C(f,1/2)$. This definition is very similar to the one given in Iwaniec and Kowalski \cite{IK} and only differs by a constant factor to the power of the degree of the $L$-function.
To help orient the reader, we give an example in the form of the analytic conductor of a Dirichlet $L$-function.  Let  $\chi$ be a Dirichlet character modulo $q$ then the associated $L$-function has analytic conductor:
\[C(\chi)=q\frac{1+\mathfrak{a}}{2\pi}, \text{ where }\mathfrak{a}=\begin{cases}
1 & \text{ if } \chi(-1)=-1 \\
0 & \text{if } \chi(-1)=1 .
\end{cases}
\] 

\section{Results}
Here we detail the theorems and make some remarks about how they fit into the general context of what is already known.
\begin{thm}\label{mainup}
 Let $d\ge 1$ be a fixed positive integer and let $L(s,f)$ be an $L$-function of degree $d$ with conductor $q(f)$ and analytic conductor $C(f)$. Suppose that GRH and Ramanujan-Petersson  hold for $L(s,f)$. 
 Then for $C(f)$ chosen such that 
$\log C(f)\ge23d$ we have 
\begin{align*}
|L(1,f)|&\le 2^de^{d\g}\left((\log\log C(f)-\log2d)^d+\frac{d}2(\log\log C(f)-\log2d)^{d-1}+\frac{dK(d)}4(\log\log C(f)-\log2d)^{d-2}\right),
\end{align*}
where 
\begin{equation}\label{defKd}
K(d)=2.31+\frac{22.59}{d}(e^{0.31d}-1-0.31d).
\end{equation}
\end{thm}
 %
\begin{xrem}\label{comparisonrem} This result is asymptotically better than the classical bound as it has the shape
 \[|L(1,f)|\le (2e^{\g})^d\left((\log\log C(f))^d-(d\log(2d)-\tfrac{d}2)(\log\log C(f))^{d-1}+O_d((\log\log C(f))^{d-2})\right),\]
 and $(d\log(2d)-\tfrac{d}2)>0$ for all $d\ge 1$.  
 \end{xrem}
\begin{xrem} If we take $d=1$, we may take $C(f)\ge 10^{10}$ and we obtain $K(1)/4\le 0.88$ which gives essentially Theorem \ref{mainthmlamlisound}
\begin{align*}
|L(1,\chi)|\le 2e^{\gamma}\left(\log\log C(f)-\log2+\frac12+\frac{0.88}{\log\log C(f)-\log 2}\right).\\
\end{align*}
\end{xrem}
\begin{thm}\label{mainlow}
Let $d\ge 1$ be a fixed positive integer and let $L(s,f)$ be an $L$-function of degree $d$ with conductor $q(f)$ and analytic conductor $C(f)$. Suppose that GRH and Ramanujan-Petersson  hold for $L(s,f)$. 
 Then for $C(f)$ chosen such that 
$\log C(f)\ge23d$ we have 
\begin{multline*}
 \frac1{|L(1,f)|}\le \left(\frac{12e^{\g}}{\pi^2}\right)^d\left((\log\log C(f)-\log2d)^d+\frac{d}2(\log\log C(f)-\log2d)^{d-1}\right.\\
 \left.+\frac{dJ_1(d)}4(\log\log C(f)-\log2d)^{d-2}+\frac{d^2J_2(d)(\log\log C(f)-\log2d)^d}{\log C(f)}\right)
\end{multline*}
where
\begin{equation}\label{defJ1d}
J_1(d)\le2+\frac{4.18}{d}(e^{0.69d}-1-0.69d).
\end{equation}
 and
\begin{equation}\label{defJ2d}
J_2(d)=9+\frac{16.74}{d}(e^{0.69d}-1-0.69d).
\end{equation}
\end{thm}
We notice that lower bound also provides something asymptotically better as in Remark \ref{comparisonrem}.
\begin{xrem} If $d=1$ we may take $C(f)\ge 10^{10}$ then $J_1(1)/4\le 0.82$ and  $J_2(1)\le 14.09$ this provides essentially Theorem \ref{mainthmlamlisound}
\[\frac1{|L(1,\chi)|}\le \frac{12e^{\g}}{\pi^2}\left(\log\log C(f)+\frac12-\log2+\frac{0.82}{\log\log C(f)-\log2}+\frac{14.09(\log\log C(f)-\log2)}{\log C(f)}\right).\]
\end{xrem}
As an easy corollary to these theorems we may a bound degree $d$ $L$-functions in the $t$ aspect as follows. 
Let $t$ be a real number and define $L_t(1,f):=L(1+it,f)$, then the analytic conductor of $L_t(s,f)$ is given by
\[C_t(f):= \frac{q(f)}{\pi^d}\prod_{j=1}^d\left|\frac{1+it+\kappa_j}{2}\right|\asymp_f |t|^d.\]
\begin{cor}\label{cor2.3}
Let $d\ge 1$ be a fixed positive integer and let $L(s,f)$ be an $L$-function of degree $d$ with conductor $q(f)$ and analytic conductor $C(f)$. Suppose that GRH and Ramanujan-Petersson  hold for $L(s,f)$.  If $\log C_t(f)\ge 23d$ then 
\[|L(1+it,f)|\le (2e^{\g})^d\left((\log\log C_t(f)-\log2d)^d+\frac{d}2(\log\log C_t(f)-\log2d)^{d-1}+\frac{dK(d)}4(\log\log C_t(f)-\log2d)^{d-2}\right),\]
and
\begin{multline*}
 \frac1{|L(1+it,f)|}\le \left(\frac{12e^{\g}}{\pi^2}\right)^d\left((\log\log C_t(f)-\log2d)^d+\frac{d}2(\log\log C_t(f)-\log2d)^{d-1}\right.\\
 \left.+\frac{dJ_1(d)}4(\log\log C_t(f)-\log2d)^{d-2}+\frac{d^2J_2(d)(\log\log C_t(f)-\log2d)^d}{\log C_t(f)}\right).
 \end{multline*}
 The definitions of $K(d)$, $J_1(d)$ and $J_2(d)$ are given by equations $\eqref{defKd}$, $\eqref{defJ1d}$ and $\eqref{defJ2d}$ respectively.
\end{cor}

\section{Lemmata}
In this section we will outline a number of results which are necessary for proving the final bound. Additionally, we will disclose a few more properties of the $L$-functions we are studying. 
First, the logarithmic derivative of $L(s,f)$ is given by 
$$
-\frac{L'}{L}(s,f)=\sum_{n\geq 2}\frac{a_{f}(n)\Lambda(n)}{n^s} \text{ for } \tRe(s)>1,
$$
where $a_f(n)=0$ unless $n=p^k$ is a prime power in which case $a_f(n)=\sum_{j=1}^d\alpha_{j,f}(p)^k$. Since $L(s,f)$ satisfies the Ramanujan-Petersson conjecture, then $|a_f(n)|\leq d$.
Further, let  $\{\rho_f\}$ be the set of the nontrivial zeros of $L(s,f)$.
 Then we have the Hadamard factorization formula (\cite[Theorem 5.6]{IK}),
\begin{equation}\label{HadL1}
 \xi(s,f)= e^{A(f)+sB(f)}\prod_{\rho_f}\left(1-\frac{s}{\rho_f}\right)e^{s/\rho_f},
\end{equation}
where $A(f)$ and $B(f)$ are constants. We note that $\tRe B(f)= -\tRe\sum_{\rho_f}1/\rho_f$ and 
taking the logarithmic derivatives of both sides of \eqref{HadL1} gives
\begin{equation}\label{HadL2}
\tRe\frac{\xi'}{\xi}(s,f)=\tRe\sum_{\rho_f}\frac{1}{s-\rho_f}.
\end{equation}

\subsection{Explicit Formulas for $\log|L(1,f)|$ and $|\tRe(B(f))|$.}
\begin{lem}\label{boundlogL1}
 Let $d\ge 1$ be a fixed positive integer and let $L(s,f)$ be an $L$-function of degree $d$ with conductor $q(f)$. Suppose that GRH and Ramanujan-Petersson  hold for $L(s,f)$.  For any $x\ge2$ there exists a real number $|\theta|\le1$ such that 
\begin{align*}
 \log|L(1,f)|&=\tRe\sum_{n\le  x}\frac{a_f(n)\Lambda(n)}{n\log n}\frac{\log(\tfrac{x}n)}{\log x}+\frac1{2\log x}\left(\log\frac{q(f)}{\pi^d}+\tRe\sum_{j=1}^d\frac{\G'}{\G}\left(\frac{1+\kappa_j}2\right)\right)\\&
 -\left(\frac{1}{\log x}-\frac{2\theta}{\sqrt{x}\log^2x}\right)|\tRe B(f)| +\frac{2d\theta}{x\log^2x}.
\end{align*}

 \end{lem}
\begin{proof}
 We have for any fixed $\sigma\ge 1$ that 
 \[\frac1{2\pi i}\int_{2-i\infty}^{2+i\infty}-\frac{L'}L(s+\s,f)\frac{x^s}{s^2}ds=\sum_{n\le x}\frac{a_f(n)\Lambda(n)}{n^{\s}}\log(\tfrac{x}n).\]
 Shifting the contour to the left, we see this integral is also equal to 
 \[-\left(\frac{L'}{L}\right)'(\s,f)-\frac{L'}{L}(\s,f)\log x-\sum_{\rho_f}\frac{x^{\rho_f-\s}}{(\rho_f-\s)^2}-\sum_{j=1}^d\sum_{m=0}^{\infty}\frac{x^{-2m-\kappa_j-\s}}{(2m+\kappa_j+\s)^2}.\]
 Thus, since $\tRe(\kappa_j)\ge0$ we have
 \begin{align*}
  -\frac{L'}{L}(\s,f)&=\sum_{n\le x}\frac{a_f(n)\Lambda(n)}{n^{\s}}\frac{\log(\tfrac{x}n)}{\log x}+\frac1{\log x}\left(\frac{L'}{L}\right)'(\s,f)\\
  &+\frac{\theta x^{\tfrac12-\s}}{\log x}\sum_{\rho_f}\frac1{|\rho_f|^2}+\frac{\theta  x^{-\s}}{\log x}\sum_{j=1}^d\sum_{m=0}^{\infty}\frac{x^{-2m}}{(2m+1)^2}.
 \end{align*}
We integrate both sides with respect to $\s$ from $1$ to $\infty$, then take real parts to obtain
\begin{align*}
 \log|L(1,f)|=\tRe\sum_{n\le  x}\frac{a_f(n)\Lambda(n)}{n\log n}\frac{\log(\tfrac{x}n)}{\log x}-\frac1{\log x}\tRe \frac{L'}{L}(1,f)+\frac{\theta}{\sqrt{x}\log^2x}\sum_{\rho_f}\frac1{|\rho_f|^2}+\frac{2d\theta}{x\log^2x}.
 \end{align*}
We note that $\sum_{\rho_f}\frac1{|\rho_f|^2}=2|\tRe B(f)|$. 
Now we have 
\begin{align*}
 -\frac{L'}{L}(1,f)&=\frac12\log q(f)-\frac{d}{2}\log\pi+\frac12\sum_{j=1}^d\frac{\G'}{\G}\left(\frac{1+\kappa_j}{2}\right)-\frac{\xi'}{\xi}(1,f)\\
\end{align*}
%
%
Hence, after taking real parts we have the desired result.
\end{proof}
\begin{lem}\label{boundforReBf}
 Let $d\ge 1$ be a fixed positive integer and let $L(s,f)$ be an $L$-function of degree $d$ with conductor $q(f)$. Suppose that GRH and Ramanujan-Petersson  hold for $L(s,f)$.  Define $0\le l(f)\le d$ to be the number of $\kappa_j$ in the gamma factor of $L(s,f)$ which equal $0$. 
 For any $x>1$ there exists a real number $|\theta|\le1$ such that 
 \begin{multline*}
  -\frac{\xi'}{\xi}(0,\overline{f})-\frac1{x}\frac{\xi'}{\xi}(0,f)+\frac{2\theta}{\sqrt{x}}|\tRe(B(f))|=\\
  \frac1{2}\log\left(\frac{q(f)}{\pi^d}\right)\left(1-\frac1x\right)-\sum_{n\le x}\frac{a_f(n)\Lambda(n)}{n}\left(1-\frac{n}x\right)+E(f,x),
 \end{multline*}
where 
\begin{multline*}
 E(f,x)=l(f)\left(-\log2-\frac{\g}2\left(1-\frac1x\right)+\frac{\log x+1}x -\sum_{n=1}^{\infty}\frac{x^{-2n-1}}{2n(2n+1)}\right)\\
 +\sum_{i=1}^{d-l(f)}\left(\frac1{2}\frac{\G'}{\G}(\frac{1+\kappa_i}2)-\frac1{2x}\frac{\G'}{\G}(\frac{\kappa_i}2)-\sum_{n=0}^{\infty}\frac{x^{-2n-\kappa_j-1}}{(2n+\kappa_i)(2n+\kappa_i+1)}\right),
\end{multline*}
In particular, $\left(1+\frac1{x}+\frac{2\theta}{\sqrt{x}}\right)|\tRe(B(f)|$ equals
\begin{align*}
\frac12\left(1-\frac1x\right)\left(\log\left(\frac{q(f)}{\pi^d}\right)+\tRe\sum_{j=1}^{d}\frac{\G'}{\G}\left(\frac{1+\kappa_j}2\right)\right)-\tRe\sum_{n\le x}\frac{a_f(n)\Lambda(n)}{n}\left(1-\frac{n}x\right)\\-(d\theta-(1+\theta)l(f))\sum_{n=1}^{\infty}\frac{x^{-2n-1}}{2n(2n+1)}+l(f)\frac{\log x+1}{x}
+\frac{(d-2l(f))\log2}x\\-\frac1{x}\sum_{i=1}^{d-l(f)}\tRe\left(\sum_{n=1}^{\infty}\left(\frac2{\kappa_i+1+2n}-\frac1{\kappa_i+1+n}\right)+\frac{x^{-\kappa_i}-1}{\kappa_i(\kappa_i+1)}\right).
\end{align*}
%
In both of the above expressions, the terms inside $\sum_{i=1}^{d-l(f)}$ are ranging over the local parameters at infinity, $\kappa_i\neq 0$.
\end{lem}
\begin{proof}
 We consider 
 \begin{equation*}
  I(f)=\frac1{2\pi i}\int_{2-i\infty}^{2+i\infty}\frac{\xi'}{\xi}(s,f)\frac{x^{s-1}}{s(s-1)}ds.
 \end{equation*}
Pulling the contour to the left we collect the residues of the poles at $s=0,1$ and $\rho_f$ the nontrivial zeros of $L(s,f)$. Hence,
\begin{align*}
  I(f)=\frac{\xi'}{\xi}(1,f)-\frac1x\frac{\xi'}{\xi}(0,f)+\sum_{\rho_f}\frac{x^{\rho_f-1}}{\rho_f(\rho_f-1)}.
 \end{align*}
Thus applying GRH we have for some $|\theta|\le 1$
\begin{equation*}
 I(f)=-\frac{\xi'}{\xi}(0,\overline{f})-\frac1x\frac{\xi'}{\xi}(0,f)+\frac{2\theta}{\sqrt{x}}|\tRe(B(f))|.
\end{equation*}
On the other hand, we can also write 
\begin{align*}
I(f)&=\frac1{2\pi i}\int_{2-i\infty}^{2+i\infty}\left(\frac12\log(q(f))+\frac{\g'}{\g}(s,f)+\frac{L'}{L}(s,f)\right)\frac{x^{s-1}}{s(s-1)}ds\\
&=\frac1{2\pi i}\int_{2-i\infty}^{2+i\infty}\frac12\log\left(\frac{q(f)}{\pi^d}\right)\frac{x^{s-1}}{s(s-1)}ds+\frac1{2\pi i}\int_{2-i\infty}^{2+i\infty}\frac12\sum_{j=1}^d\frac{\G'}{\G}\left(\frac{s+\kappa_j}2\right)\frac{x^{s-1}}{s(s-1)}ds\\
&\quad+\frac1{2\pi i}\int_{2-i\infty}^{2+i\infty}\frac{L'}{L}(s,f)\frac{x^{s-1}}{s(s-1)}ds\\
&=I_1+I_2+I_3.
\end{align*}
The contribution from $I_1$ and $I_3$ is 
\[\frac12\log\left(\frac{q(f)}{\pi^d}\right)\left(1-\frac1x\right)-\sum_{n\le x}\frac{a_f(n)\Lambda(n)}{n}\left(1-\frac{n}x\right).\]
We rewrite $I_2$ as 
\[I_2=\sum_{j=1}^d\frac1{2\pi i}\int_{2-i\infty}^{2+i\infty}\frac12\frac{\G'}{\G}\left(\frac{s+\kappa_j}2\right)\frac{x^{s-1}}{s(s-1)}ds.\]
Fix $j$, if $\kappa_j\neq0$ then the $j$-th term of the summand will have simple poles at 
$s=0,1$ and $s=-2n-\kappa_j$ for $n\ge0 $. Thus the contribution will be 
\[\frac12\frac{\G'}{\G}\left(\frac{1+\kappa_j}2\right)-\frac1{2x}\frac{\G'}{\G}\left(\frac{\kappa_j}2\right)-\sum_{n=0}^{\infty}\frac{x^{-2n-\kappa_j-1}}{(2n+\kappa_j)(2n+1+\kappa_j)}.\]
On the other hand, if $\kappa_j=0$ then the $j$-th term of the summand 
will have simple poles at $s=1$ and $s=-2n$ for $n\ge 1$, which contribute
\[\frac12\frac{\G'}{\G}\left(\frac12\right)-\sum_{n=1}^{\infty}\frac{x^{-2n-1}}{2n(2n+1)}.\]
Additionally, we know that 
\[\frac12\frac{\G'}{\G}\left(\frac{s}2\right)=-\frac1s+\frac12\frac{\G'}{\G}\left(\frac{s}2+1\right),\]
so the residue of the double pole  at $s=0$ is given by
\[\frac{1+\log x+\frac12\frac{\G'}{\G}(1)}{x}.\]
Using the fact that $\frac{\G'}{\G}(1)=-\g$ and  $\frac{\G'}{\G}(1/2)=-2\log2-\gamma$ we see the overall contribution will be 
%
\[-\log 2 -\frac{\g}2\left(1-\frac1x\right)+\frac{\log x+1}x-\sum_{n=1}^{\infty}\frac{x^{-2n-1}}{2n(2n+1)}.\]
Let $l(f)$ be as in the statement of the lemma. Then, reordering the $\kappa_j$ so that $\kappa_1,\kappa_2,\ldots,\kappa_{d-l(f)}$ are all nonzero and summing over $j$ we get the desired expression for $E(f,x)$.\\
Finally, since $-\tRe\frac{\xi'}{\xi}(0,\overline{f})=-\tRe\frac{\xi'}{\xi}(0,f)=|\tRe(B(f))|$, we see that taking real parts of the established identity we obtain 
\[\left(1+\frac1x+\frac{2\theta}{\sqrt{x}}\right)|Re(B(f))|=\frac12\log\left(\frac{q(f)}{\pi^d}\right)\left(1-\frac1x\right)-\tRe\sum_{n\le x}\frac{a_f(n)\Lambda(n)}{n}\left(1-\frac{n}x\right)+\tRe(E(f,x)).\] 
We find an explicit expression for the right hand side as follows: \\
Start by noting that for $\kappa_j=0$ we have 
\[\frac{\G'}{\G}\left(\frac12\right)=\frac{\G'}{\G}\left(\frac{1+\kappa_j}2\right),\]
so that 
\begin{align*}
 &\frac12\log\left(\frac{q(f)}{\pi^d}\right)\left(1-\frac1x\right)+E(f,x)=\frac12\log\left(\frac{q(f)}{\pi^d}\right)\left(1-\frac1x\right)+\sum_{j=1}^{d}\frac1{2}\frac{\G'}{\G}\left(\frac{1+\kappa_j}2\right)\\
 &+l(f)\left(\frac{\log x+1+\gamma/2}x -\sum_{n=1}^{\infty}\frac{x^{-2n-1}}{2n(2n+1)}\right)
 -\sum_{i=1}^{d-l(f)}\left(\frac1{2x}\frac{\G'}{\G}\left(\frac{\kappa_i}2\right)+\sum_{n=0}^{\infty}\frac{x^{-2n-\kappa_i-1}}{(2n+\kappa_i)(2n+\kappa_i+1)}\right).
\end{align*}
We note that for some $|\theta|\le 1$
\[\sum_{n=0}^{\infty}\frac{x^{-2n-\kappa_j-1}}{(2n+\kappa_i)(2n+\kappa_i+1)}=\frac{x^{-\kappa_i}}{x\kappa_i(\kappa_j+1)}+\theta\sum_{n=1}^{\infty}\frac{x^{-2n-1}}{2n(2n+1)},\]
hence 
\begin{align}\label{predigammabnd}
 &\frac12\log\left(\frac{q(f)}{\pi^d}\right)\left(1-\frac1x\right)+E(f,x)=\frac12\log\left(\frac{q(f)}{\pi^d}\right)\left(1-\frac1x\right)+\sum_{j=1}^{d}\frac1{2}\frac{\G'}{\G}\left(\frac{1+\kappa_j}2\right)\\
\nonumber&-(d\theta-(1+\theta)l(f))\sum_{n=1}^{\infty}\frac{x^{-2n-1}}{2n(2n+1)}+l(f)\frac{\log x+\g/2+1}{x}-\frac1{x}\sum_{i=1}^{d-l(f)}\left(\frac12\frac{\G'}{\G}\left(\frac{\kappa_i}{2}\right)+\frac{x^{-\kappa_i}}{\kappa_i(\kappa_i+1)}\right).
\end{align}
Now, from the functional equation of $\Gamma(s)$ we see that 
\[\frac12\frac{\G'}{\G}\left(\frac{\kappa_i}{2}\right)=\frac12\frac{\G'}{\G}\left(\frac{\kappa_i}2+1\right)-\frac1{\kappa_i},\]
we recall Legendre's duplication formula 
\[\Gamma(s)\Gamma(s+\tfrac12)=2^{1-2s}\log(\sqrt{\pi})\Gamma(2s),\]
so we have 
\[\frac12\frac{\G'}{\G}\left(\frac{\kappa_i}2+1\right)=-\log2+\frac{\G'}{\G}(\kappa_i+1)-\frac12\frac{\G'}{\G}\left(\frac{\kappa_i+1}{2}\right).\]
Finally, we note 
\[\frac{\G'}{\G}(s)=-\g-\frac{1}s-\sum_{n=1}^{\infty}\left(\frac1{s+n}-\frac1{n}\right),\]
so that
\[\left[\frac{\G'}{\G}(\kappa_i+1)-\frac12\frac{\G'}{\G}\left(\frac{\kappa_i+1}{2}\right)\right]-\frac12\frac{\G'}{\G}\left(\frac{\kappa_i+1}{2}\right)=\frac1{\kappa_i+1}+\sum_{n=1}^{\infty}\frac2{\kappa_i+1+2n}-\frac1{\kappa_i+1+n}.\]
Combinging these facts gives \eqref{predigammabnd} as 
\begin{align}\label{predigammabnd2}
&\frac12\log\left(\frac{q(f)}{\pi^d}\right)\left(1-\frac1x\right)+E(f,x)=\frac12\log\left(\frac{q(f)}{\pi^d}\right)\left(1-\frac1x\right)+\sum_{j=1}^{d}\frac1{2}\frac{\G'}{\G}\left(\frac{1+\kappa_j}2\right)\\
\nonumber&-(d\theta-(1-\theta)l(f))\sum_{n=1}^{\infty}\frac{x^{-2n-1}}{2n(2n+1)}+l(f)\frac{\log x+\g/2+1}{x}\\
\nonumber&-\frac1{x}\sum_{i=1}^{d-l(f)}\left(\frac12\frac{\G'}{\G}\left(\frac{\kappa_i+1}2\right)-\log2+\sum_{n=1}^{\infty}\left(\frac2{\kappa_i+1+2n}-\frac1{\kappa_i+1+n}\right)+\frac{x^{-\kappa_i}-1}{\kappa_i(\kappa_i+1)}\right).
\end{align}
Then since $\frac12\frac{\G'}{\G}\left(\frac{\kappa_i+1}2\right)=\frac12\frac{\G'}{\G}(\frac{1}2)=-\log 2 -\g/2$ when $\kappa_i=0$ we add $\frac{-l(f)\frac{\G'}{\G}(\frac12)+l(f)\frac{\G'}{\G}(\frac12)}{2x}$ so that the RHS of \eqref{predigammabnd2} is given by
\begin{align*}
\frac12\left(1-\frac1x\right)\left(\log\left(\frac{q(f)}{\pi^d}\right)+\sum_{j=1}^{d}\frac{\G'}{\G}\left(\frac{1+\kappa_j}2\right)\right)-(d\theta-(1+\theta)l(f))\sum_{n=1}^{\infty}\frac{x^{-2n-1}}{2n(2n+1)}+l(f)\frac{\log x+1}{x}\\
+\frac{(d-2l(f))\log2}x-\frac1{x}\sum_{i=1}^{d-l(f)}\left(\sum_{n=1}^{\infty}\left(\frac2{\kappa_i+1+2n}-\frac1{\kappa_i+1+n}\right)+\frac{x^{-\kappa_i}-1}{\kappa_i(\kappa_i+1)}\right).
\end{align*}
Taking real parts gives the desired result. 
\end{proof}
\subsection{Bounds for the Digamma Function}
The following are some technical lemmas which help to shorten the proof of the main results. The first is taken from V. Chandee. 

\begin{lem}\cite[Lemma 2.3]{Chandee}\label{Chandee}
Let $z=x+iy$, where $x\ge \frac14$. Then 
\[\tRe\frac{\G'}{\G}(z)\le \log|z|.\]
\end{lem}

\begin{lem}\label{techlem1}
Let $\kappa=\sigma+it$ such that $\s\ge 0$, then 
\begin{align*}
\tRe\left(\sum_{n=1}^{\infty}\left(\frac2{\kappa+1+2n}-\frac1{\kappa+1+n}\right)\right)&=\frac12\log4+\frac{\s^2+3\s+2+t^2}{(\s+2)^2+t^2}-\frac{\s^2+4\s+3+t^2}{(\s+3)^2+t^2}+\frac12\log\left(\frac{(\s+2)^2+t^2}{(\s+3)^2+t^2}\right).
\end{align*}
 \end{lem}
\begin{proof}
We take the real part inside the sum and focus on the individual partial sums given by
\[\sum_{n=1}^N\frac{2(\s+1+2n)}{(\s+1+2n)^2+t^2} \text{ and } \sum_{n=1}^N\frac{(\s+1+n)}{(\s+1+n)^2+t^2}. \]
Using partial summation we find 
\begin{align*}
\sum_{n=1}^N\frac{2(\s+1+2n)}{(\s+1+2n)^2+t^2}&=\frac{2N(\s+1+2N)+\s^2+2\s(N+1)+2Y+1+t^2}{(\s+1+2N)^2+t^2}-\frac{\s^2+4\s+3+t^2}{(\s+3)^2+t^2}\\
&+\frac12\log((\s+1+2N)^2+t^2)-\frac12\log((\s+3)^2+t^2) \\
\text{ and }&\,\\
\sum_{n=1}^N\frac{(\s+1+n)}{(\s+1+n)^2+t^2}&=\frac{N(\s+1+N)+\s^2+\s(N+1)+\s+N+1+t^2}{(\s+1+N)^2+t^2}-\frac{\s^2+3\s+2+t^2}{(\s+2)^2+t^2}\\
&+\frac12\log((\s+1+N)^2+t^2)-\frac12\log((\s+2)^2+t^2).
\end{align*}
Taking the limit as $N\to \infty$ we see 
\begin{align*}
\tRe\left(\sum_{n=1}^{\infty}\left(\frac2{\kappa+1+2n}-\frac1{\kappa+1+n}\right)\right)&=\frac12\log4+\frac{\s^2+3\s+2+t^2}{(\s+2)^2+t^2}-\frac{\s^2+4\s+3+t^2}{(\s+3)^2+t^2}+\frac12\log\left(\frac{(\s+2)^2+t^2}{(\s+3)^2+t^2}\right),
\end{align*}
as was claimed.
\end{proof}
\begin{lem}\label{techlem2}
Let $\kappa=s+it$ such that $\s\ge0$, and $x>1$ then 
\[\left|\frac{x^{-\kappa}-1}{\kappa(\kappa+1)}\right|\le\frac{2\log x}{\log3}.\]
\end{lem}
\begin{proof}
We consider two cases. \\
First suppose  $|\kappa|\ge \frac{c}{\log x}$ then we can trivially bound the norm to obtain
\[\left|\frac{x^{-\kappa}-1}{\kappa(\kappa+1)}\right|\le \frac{2\log x}c.\]
 If $|\kappa|< \frac{c}{\log x}$ then 
 \[x^{-\kappa}-1=\sum_{k=1}^{\infty}\frac{(-\kappa\log x)^k}{k!},\] so that 
 \[\left|\frac{x^{-\kappa}-1}{\kappa(\kappa+1)}\right|\le\log x\sum_{k=1}^{\infty}\frac{c^{k-1}}{k!}= \frac{e^c-1}c\log x.\]

The choice of $c=\log3 $ gives the desired result.
\end{proof}

\subsection{Relevant Results from \cite{LamLiSound}.}
Let  $$B=-\sum_{\rho}\tRe\frac1{\rho}=\frac12\log(4\pi)-1-\frac{\g}2,$$
where the sum is taken over the non-trivial zeros of the Riemann zeta function.
\begin{lem}\cite[Lemma 2.4]{LamLiSound}
\label{LamLiSound2.4}
Assume the Riemann Hypothesis. For $x>1$ we have, for some $|\theta|\le 1$, 
\[\sum_{n\le x}{\Lambda(n)}{n}\left(1-\frac{n}{x}\right)=\log x-(1+\gamma)+\frac{2\pi}x-\sum_{n=1}^{\infty}\frac{x^{-2n-1}}{2n(2n+1)}+2\frac{\theta|B|}{\sqrt{x}}.\]
\end{lem}

\begin{lem}\cite[Lemma 2.6]{LamLiSound}
\label{LamLiSound2.6}
Assume the Riemann Hypothesis. For all $x\ge e$ we have 
\[\sum_{n\le x}\frac{\Lambda(n)}{n\log n}\frac{\log(x/n)}{\log x}=\log\log x -\g-1+\frac{\g}{\log x}+\frac{2|B|\theta}{\sqrt{x}\log^2x}+\frac{\theta}{3x^3\log^2x}.\]
\end{lem}
We also prove the following lemma which is a slight generalization of \cite[Lemma 5.1]{LamLiSound}. 
\begin{lem}\label{generalizedLLS5.1}
Assume the Ramanujan-Petersson conjecture. Then for $x\ge 100$ we have 
\begin{equation}\label{ineq1}
\tRe\sum_{n\le  x}\a_{j,f}(n)\Lambda(n)\left(\frac1{n\log n}-\frac1{x\log x}\right)\ge \sum_{p^k\le  x}\Lambda(p^k)(-1)^k\left(\frac1{p^k\log p^k}-\frac1{x\log x}\right).\end{equation}
In particular, we have 
\begin{equation}\label{ineq2}
\tRe\sum_{n\le  x}a_f(n)\Lambda(n)\left(\frac1{n\log n}-\frac1{x\log x}\right)\ge d\sum_{p^k\le  x}\Lambda(p^k)(-1)^k\left(\frac1{p^k\log p^k}-\frac1{x\log x}\right).\end{equation}
\end{lem}
\begin{proof}
Note that if $x$ is a prime power then the summand at $x$ on both sides of the inequality \eqref{ineq1} contribute $0$, so we assume $x$ is not a prime power. We begin by recalling that $a_f(n)=0$ unless $n=p^k$ is a prime power in which case $a_f(n)=\sum_{j=1}^d\alpha_{j,f}(p)^k$. So that \eqref{ineq2} follows immediately, once we prove \eqref{ineq1}.\\
\indent Fix $j$ and consider each  $\alpha_{j,f}$ separately. From the definition we see $\alpha_{j,f}(n)$ is only nonzero if $n=p^k$ for some prime power. If $\alpha_{j,f}(p)=0$ then the contribution is $0$ while the value on the right hand side $<0$. If $\alpha_{j,f}(p)\neq0$ then, from Ramanujan-Petersson we have that $|\alpha_{j,f}(p)|\le 1$, so we express $\alpha_{j,f}(p)=-re(\theta)$, for $0<r\le 1$ where $e(\theta)=e^{2\pi i\theta}$.  Consider the difference of the left and right side of \eqref{ineq1}:
\begin{equation}\label{diffpos}
\log(p)\sum_{p^k\le x}(-1)^{k-1}(1-r^k\cos(k\theta))\left(\frac1{p^k\log p^k}-\frac1{x\log x}\right).
\end{equation}
If we establish this is non-negative, then we are finished.\\
Before we proceed we see that for all $k\ge 1$
\begin{equation}\label{trigineq}
1-r^k\cos(k\theta)\le k^2(1-r\cos\theta). 
\end{equation}
The case $k=1$ is trivial, for the remaining $k\ge 2$, the inequality follows from
\[k^2-1\ge 3>r^k+r\ge r^k\cos(k\theta)-r\cos(\theta).\]
If $p\ge3$, then by \eqref{trigineq} we have \eqref{diffpos} is greater than
\[\log(p)(1-r\cos\theta)\left(\frac1{p\log p}-\frac1{x\log x}-\sum_{j=1}^{\infty}\frac{(2j)^2}{p^{2j}\log p^{2j}}\right)\ge 0.\]
For $p=2$, when $k\ge 6$ we apply \eqref{trigineq} again. Otherwise, when $1\le k\le5$ we compute the trigonometric polynomial exactly. A little computer computation completes the result.
\end{proof}

\section{Proof of Theorems \ref{mainup} and \ref{mainlow}}
\subsection{Upper bounds for $L(1,f)$.}
Let $C(f)\ge 10^{10}$ and $x\ge 132$, be a real number to be chosen later. Lemma \ref{boundlogL1} says 
\begin{align*}
\log|L(1,f)|\le& \tRe\sum_{n\le x}\frac{a_f(n)\Lambda(n)}{n\log n}\frac{\log(x/n)}{\log x}+\frac1{2\log x}\left(\log\frac{q(f)}{\pi^d}+\tRe\sum_{j=1}^d\frac{\G'}{\G}\left(\frac{\kappa_i+1}2\right)\right)\\ 
&+\frac{2d}{x\log^2x}-\left(\frac1{\log x}-\frac2{\sqrt{x}\log^2 x}\right)|\tRe B(f)|.
\end{align*}
Applying Lemma \ref{boundforReBf} with the conditions on $x$ as above, we see 
\begin{multline*}
|\tRe B(f)|\ge \left(1+\frac1{\sqrt{x}}\right)^{-2}\left(\frac12\left(1-\frac1x\right)\left(\log\frac{q(f)}{\pi^d}+\tRe\sum_{j=1}^d\frac{\G'}{\G}\left(\frac{\kappa_j+1}2\right)\right)-\tRe\sum_{n\le x}\frac{a_f(n)\Lambda(n)}{n}\left(1-\frac{n}x\right)\right.\\
\left.+l(f)\frac{\log x+1}x+\frac{(d-2l(f))\log2}{x}-(d-2l(f))\sum_{n=1}^{\infty}\frac{x^{-2n-1}}{2n(2n+1)}\right.\\
\left.-\frac1x\sum_{i=1}^{d-l(f)}\tRe\left(\sum_{n=1}^{\infty}\left(\frac2{\kappa_i+1+2n}-\frac1{\kappa_i+1+n}\right)+\frac{x^{-\kappa_i}-1}{\kappa_i(\kappa_i+1)}\right)\right).
\end{multline*}
For $x\ge 132$ we bound
\begin{multline*}
-\left(\frac1{\log x}-\frac{2}{\sqrt{x}\log^2 x}\right)\left(1+\frac1{\sqrt{x}}\right)^{-2}\left(l(f)\frac{\log x+1}{x}+\frac{(d-2l(f))\log2}{x}\right.\\
\left.-(d-2l(f))\sum_{n=1}^{\infty}\frac{x^{-2n-1}}{2n(2n+1)}-\frac1x\sum_{i=1}^{d-l(f)}\tRe\left(\sum_{n=1}^{\infty}\left(\frac2{\kappa_i+1+2n}-\frac1{\kappa_i+1+n}\right)+\frac{x^{-\kappa_i}-1}{\kappa_i(\kappa_i+1)}\right)\right)+\frac{2d}{x\log^2 x}\\
=-\left(\frac1{\log x}-\frac{2}{\sqrt{x}\log^2 x}\right)\left(1+\frac1{\sqrt{x}}\right)^{-2}(A_1+A_2-A_3-A_4-A_5)+\frac{2d}{x\log^2x}.
\end{multline*}
First, we consider 
\[A_4=\frac1x\sum_{i=1}^{d-l(f)}\tRe\left(\sum_{n=1}^{\infty}\left(\frac2{\kappa_i+1+2n}-\frac1{\kappa_i+1+n}\right)\right).\]
Fix $i$ and study the inner sum, writing $\kappa_i=\sigma+it$, and noting that Ramanujan-Petersson gives us $\s\ge 0$,  we apply Lemma \ref{techlem1} so that 
\begin{align*}
\tRe\left(\sum_{n=1}^{\infty}\left(\frac2{\kappa_i+1+2n}-\frac1{\kappa_i+1+n}\right)\right)&=\frac12\log4+\frac{\s^2+3\s+2+t^2}{(\s+2)^2+t^2}-\frac{\s^2+4\s+3+t^2}{(\s+3)^2+t^2}+\frac12\log\left(\frac{(\s+2)^2+t^2}{(\s+3)^2+t^2}\right)\\
&\le \log2-\frac45(\sqrt{3}-2)\le 1.
\end{align*}
The inequality comes from the following facts. First, the last term is negative. Next, taking $\s\ge 0$, a maple calculation finds that $-\frac45(\sqrt{3}-2)$ is a global maximum for 
\[\frac{\s^2+3\s+2+t^2}{(\s+2)^2+t^2}-\frac{\s^2+4\s+3+t^2}{(\s+3)^2+t^2}.\]
Thus we may combine the terms $A_2$ and $A_4$ to obtain 
\[-\left(\frac1{\log x}-\frac2{\sqrt{x}\log^2x}\right)\left(1+\frac1{\sqrt{x}}\right)^{-2}(A_2-A_4)\le \frac{(d-l(f))(1-\log2)+l(f)\log 2}{(1+\sqrt{x})^2\log x}.\]
For $A_5$, fix $i$, then writing $\kappa_i=\sigma+it$, since we have $\s\ge 0$, we  apply Lemma \ref{techlem2} to obtain 
\[\tRe\left(\frac{x^{-\kappa_i}-1}{\kappa_j(\kappa_i+1)}\right)\le\frac{2\log x}{\log3}.\]
Thus combining $A_1$ and $A_5$ we have
\[-\left(\frac1{\log x}-\frac2{\sqrt{x}\log^2x}\right)\left(1+\frac1{\sqrt{x}}\right)^{-2}(A_1-A_5)\le \frac{(2d/\log3-l(f)(1+2/\log3))}{(1+\sqrt{x})^2}-\frac{l(f)}{(1+\sqrt{x})^2\log x}.\]
Finally, for $x\ge 132$ we have 
\begin{align*}
&-\left(\frac1{\log x}-\frac{2}{\sqrt{x}\log^2 x}\right)\left(1+\frac1{\sqrt{x}}\right)^{-2}\left(A_1+A_2-A_3-A_4-A_5\right)+\frac{2d}{x\log^2 x}\\
&\le\frac1{(1+\sqrt{x})^2}\left(2d/\log3-l(f)(1+2/\log3)+\frac{(d-2l(f))(1-\log2+\sum_{n=1}^{\infty}\frac{x^{-2n}}{2n(2n+1)})}{\log x}+\frac{2d(1+\sqrt{x})^2}{x\log^2x}\right)\\ &\le\frac{2d}{(1+\sqrt{x})^2}.
\end{align*}
Hence, 
\begin{align*}
\log|L(1,f)|\le& \tRe\sum_{n\le x}\frac{a_f(n)\Lambda(n)}{n\log n}\frac{\log(x/n)}{\log x}+\frac1{2\log x}\left(\log\frac{q(f)}{\pi^d}+\tRe\sum_{j=1}^d\frac{\G'}{\G}\left(\frac{\kappa_i+1}2\right)\right)\\ 
&+\left(\frac1{\log x}-\frac{2}{\sqrt{x}\log^2 x}\right)\left(1+\frac1{\sqrt{x}}\right)^{-2}\tRe\sum_{n\le x}\frac{a_f(n)\Lambda(n)}{n}\left(1-\frac{n}x\right)+\frac{2d}{(1+\sqrt{x})^2}\\
&-\left(\frac1{\log x}-\frac{2}{\sqrt{x}\log^2 x}\right)\left(1+\frac1{\sqrt{x}}\right)^{-2}\frac12\left(1-\frac1x\right)\left(\log\frac{q(f)}{\pi^d}+\tRe\sum_{j=1}^d\frac{\G'}{\G}\left(\frac{\kappa_j+1}2\right)\right).
\end{align*}
Next, note that 
\[0\le\frac1{2\log x}-\left(\frac1{\log x}-\frac{2}{\sqrt{x}\log^2 x}\right)\left(1+\frac1{\sqrt{x}}\right)^{-2}\frac12\left(1-\frac1x\right)\le \frac1{(\sqrt{x}+1)\log x}\left(1+\frac1{\log x}\right),\]
and Lemma \ref{Chandee} gives
\[\tRe\frac{\G'}{\G}\left(\frac{\kappa_j+1}2\right)\le\log\left|\frac{1+\kappa_i}2\right|,\]
so 
\[\log\frac{q(f)}{\pi^d}+\tRe\sum_{j=1}^d\frac{\G'}{\G}\left(\frac{\kappa_i+1}2\right)\le\log C(f).\]
Therefore, 
\begin{align*}
\log|L(1,f)|\le& \tRe\sum_{n\le x}\frac{a_f(n)\Lambda(n)}{n\log n}\frac{\log(x/n)}{\log x}+\frac{\log C(f)}{(\sqrt{x}+1)\log x}\left(1+\frac1{\log x}\right)\\ 
&+\left(\frac1{\log x}-\frac{2}{\sqrt{x}\log^2 x}\right)\left(1+\frac1{\sqrt{x}}\right)^{-2}\tRe\sum_{n\le x}\frac{a_f(n)\Lambda(n)}{n}\left(1-\frac{n}x\right)+\frac{2d}{(1+\sqrt{x})^2}.
\end{align*}
%

The right hand side of the above is largest when $a_f(p)=d$ for all $p\le x$, thus 
\begin{multline*}
\log|L(1,f)|\le d\tRe\sum_{n\le x}\frac{\Lambda(n)}{n\log n}\frac{\log(x/n)}{\log x}+\frac{d}{\log x}\tRe\sum_{n\le x}\frac{\Lambda(n)}{n}\left(1-\frac{n}x\right)
+\frac{\log C(f)}{(\sqrt{x}+1)\log x}\left(1+\frac{1}{\log x}\right)+\frac{2d}{(1+\sqrt{x})^2}.
\end{multline*}
%
So applying Lemmas \ref{LamLiSound2.4} and \ref{LamLiSound2.6} and choosing $x=\frac{\log^2C(f)}{4d^2}$ (which implies $\frac{\log C(f)}{\sqrt{x}}=2d$ and allows us to factor $d$ from each term) we obtain 
\begin{equation*}
\log|L(1,f)|\le d\left(\log\log x+\g-\frac{1}{\log x}+\frac{2}{(1+\sqrt{x})^2}\right)+\frac{\log C(f)}{\sqrt{x}\log x}\left(1+\frac{1}{\log x}\right).
\end{equation*}
Thus for $x\ge 132$ we have
\begin{align*}
\log|L(1,f)|&\le d\left(\log\log x+\g+\frac1{\log x}+\frac{2}{\log^2x}+\frac{2}{(1+\sqrt{x})^2}\right)\\
&\le d\left(\log\log x+\g+\frac1{\log x}+\frac{2.31}{\log^2x}\right).
\end{align*}
Therefore, 
\[|L(1,f)|\le e^{d\g }\log^dx\left(1+\frac{d}{\log x}+\frac{dK(d)}{\log^2x}\right)\]
where $K(d)=2.31+(1+\frac{4.62}{\log x}+\frac{(2.31)^2}{\log^2x})\sum_{k=0}^{\infty}\frac{d^{k+1}}{(k+2)!}(\frac1{\log x}+\frac{2.31}{\log^2x})^k$.  Replacing $x$ gives 
\begin{align}\label{upbndprelim}
\nonumber|L(1,f)|&\le 2^de^{d\g}\left((\log\log C(f)-\log2d)^d+\frac{d}2(\log\log C(f)-\log2d)^{d-1}+\frac{dK(d)}4(\log\log C(f)-\log2d)^{d-2}\right),
\end{align}
which proves the result.
\subsection{Lower bounds for $L(1,f)$}
The argument proceeds similarly. As before we let $C(f)$ be chosen such that  $x=\frac{\log^2 C(f)}{4d^2}\ge 132$, then from Lemma \ref{boundlogL1} we have 
\begin{align*}
 \log|L(1,f)|&\ge\tRe\sum_{n\le  x}\frac{a_f(n)\Lambda(n)}{n\log n}\frac{\log(\tfrac{x}n)}{\log x}+\frac1{2\log x}\left(\log\frac{q(f)}{\pi^d}+\tRe\sum_{j=1}^d\frac{\G'}{\G}\left(\frac{1+\kappa_j}2\right)\right)\\&
 -\left(\frac{1}{\log x}+\frac{2}{\sqrt{x}\log^2x}\right)|\tRe B(f)| -\frac{2d}{x\log^2x}.
\end{align*}
Applying Lemma \ref{boundforReBf} we see 
\begin{multline*}
|\tRe(B(f))|\le \left(1-\frac1{\sqrt{x}}\right)^{-2}\left(\frac12\left(1-\frac1x\right)\left(\log\left(\frac{q(f)}{\pi^d}\right)+\tRe\sum_{j=1}^{d}\frac{\G'}{\G}\left(\frac{1+\kappa_j}2\right)\right)-\tRe\sum_{n\le x}\frac{a_f(n)\Lambda(n)}{n}\left(1-\frac{n}x\right)\right.\\
\left. -d\sum_{n=1}^{\infty}\frac{x^{-2n-1}}{2n(2n+1)}+l(f)\frac{\log x+1}{x}
+\frac{(d-2l(f))\log2}x-\frac1{x}\sum_{i=1}^{d-l(f)}\tRe\left(\sum_{n=1}^{\infty}\left(\frac2{\kappa_i+1+2n}-\frac1{\kappa_i+1+n}+\frac{x^{-\kappa_i}-1}{\kappa_i(\kappa_i+1)}\right)\right)\right).
\end{multline*}

For $x\ge 132$ we bound 
\begin{multline*}
-\left(\frac{1}{\log x}+\frac{2}{\sqrt{x}\log^2x}\right)\left(1-\frac1{\sqrt{x}}\right)^{-2}\left(l(f)\frac{\log x+1}{x}
+\frac{(d-2l(f))\log2}x-d\sum_{n=1}^{\infty}\frac{x^{-2n-1}}{2n(2n+1)}\right.\\
\left. -\frac1x\sum_{i=1}^{d-l(f)}\tRe\left(\sum_{n=1}^{\infty}\left(\frac2{\kappa_i+1+2n}-\frac1{\kappa_i+1+n}\right)\right) +\frac{x^{-\kappa_i}-1}{\kappa_i(\kappa_i+1)}\right)-\frac{2d}{x\log^2x}\\
=-\left(\frac{1}{\log x}+\frac{2}{\sqrt{x}\log^2x}\right)\left(1-\frac1{\sqrt{x}}\right)^{-2}(A_1+A_2-A_3-A_4-A_5)-\frac{2d}{x\log^2x}.
\end{multline*}
First, we consider 
\[A_4=\frac1x\sum_{i=1}^{d-l(f)}\tRe\left(\sum_{n=1}^{\infty}\left(\frac2{\kappa_i+1+2n}-\frac1{\kappa_i+1+n}\right)\right).\]
Fix $i$ and study the inner sum, writing $\kappa_i=\sigma+it$, and noting that Ramanujan-Petersson gives us $\s\ge 0$,  we apply Lemma \ref{techlem1} so that 
\begin{align*}
\tRe\left(\sum_{n=1}^{\infty}\left(\frac2{\kappa_i+1+2n}-\frac1{\kappa_i+1+n}\right)\right)&=\frac12\log4+\frac{\s^2+3\s+2+t^2}{(\s+2)^2+t^2}-\frac{\s^2+4\s+3+t^2}{(\s+3)^2+t^2}+\frac12\log\left(\frac{(\s+2)^2+t^2}{(\s+3)^2+t^2}\right)\\
&\ge 2\log2-\log(3).
\end{align*}
The inequality comes from the following facts. First, the combination of the second and third term is positive since $\s\ge 0$, and the last term has a global minimum at the point $(0,0)$ which gives $\log(2/3)$. 
Thus we may combine the terms $A_2$ and $A_4$ to obtain 
\[-\left(\frac1{\log x}+\frac2{\sqrt{x}\log^2x}\right)\left(1-\frac1{\sqrt{x}}\right)^{-2}(A_2-A_4)\ge 1.04\frac{d(\log 2-\log 3)+l(f)\log 3}{(\sqrt{x}-1)^2\log x}.\]
For $A_5$, fix $j$, then writing $\kappa_j=\sigma+it$ and invoking Ramanujan-Petersson, we can apply Lemma \ref{techlem2} to obtain 
\[\tRe\left(\frac{x^{-\kappa_j}-1}{\kappa_j(\kappa_j+1)}\right)\ge-\frac{2\log x}{\log3}.\]
Thus combining the terms $A_1$ and $A_5$ we have
\[-\left(\frac1{\log x}+\frac2{\sqrt{x}\log^2x}\right)\left(1-\frac1{\sqrt{x}}\right)^{-2}(A_1-A_5)\ge -1.04\left(\frac{2d/\log3+l(f)(2/\log(3)-1)}{(\sqrt{x}-1)^2}+\frac{l(f)}{(\sqrt{x}-1)^2\log x}\right).\]
Finally, for $x\ge 132$ we have 
\begin{align*}
-\left(\frac{1}{\log x}+\frac{2}{\sqrt{x}\log^2x}\right)\left(1-\frac1{\sqrt{x}}\right)^{-2}(A_1+A_2-A_3-A_4-A_5)-\frac{2d}{x\log^2x}\\
\ge \frac1{(\sqrt{x}-1)^2}1.04\left(-(2d/\log 3-l(f)(2d/\log(3)-1))-\frac{2d}{1.04x\log^2x}(\sqrt{x}-1)^2\right. \\
\left. +\frac{d(\log2-\log3)+l(f)(\log(3)-1)-d\sum_{n=1}^{\infty}\frac{x^{-2n-1}}{2n(2n+1)}}{\log x}\right) \ge \frac{-2.05d}{(\sqrt{x}-1)^2}.
\end{align*}

Thus, for $x\ge 132$ 
\begin{multline*}
\log|L(1,f)|\ge\tRe\sum_{n\le  x}\frac{a_f(n)\Lambda(n)}{n\log n}\frac{\log(\tfrac{x}n)}{\log x}+\frac1{2\log x}\left(\log\frac{q(f)}{\pi^d}+\tRe\sum_{j=1}^d\frac{\G'}{\G}\left(\frac{1+\kappa_j}2\right)\right)\\
 -\left(\frac{1}{\log x}+\frac{2}{\sqrt{x}\log^2x}\right)\left(1-\frac1{\sqrt{x}}\right)^{-2}\left(\frac12\left(1-\frac1x\right)\left(\log\left(\frac{q(f)}{\pi^d}\right)+\tRe\sum_{j=1}^{d}\frac{\G'}{\G}\left(\frac{1+\kappa_j}2\right)\right)\right. \\
 \left.-\tRe\sum_{n\le x}\frac{a_f(n)\Lambda(n)}{n}\left(1-\frac{n}x\right)\right)-\frac{2.05 d}{(\sqrt{x}-1)^2}.
\end{multline*}
We note that 
\begin{multline*}
\left(\log\frac{q(f)}{\pi^d}+\tRe\sum_{j=1}^d\frac{\G'}{\G}\left(\frac{1+\kappa_j}2\right)\right)\left(\frac1{2\log x}-\left(\frac{1}{\log x}+\frac{2}{\sqrt{x}\log^2x}\right)\left(1-\frac1{\sqrt{x}}\right)^{-2}\left(\frac12\left(1-\frac1x\right)\right)\right)\\
\ge -\left(\log\frac{q(f)}{\pi^d}+\tRe\sum_{j=1}^d\frac{\G'}{\G}\left(\frac{1+\kappa_j}2\right)\right)\frac1{(\sqrt{x}-1)\log x}\left(1+\frac{1+1/\sqrt{x}}{\log x}\right)\\
\ge -\frac{\log C(f)}{(\sqrt{x}-1)\log x}\left(1+\frac{1+1/\sqrt{x}}{\log x}\right). 
\end{multline*}
Where the last inequality comes from Lemma \ref{Chandee}.
So far, we have proven 
\begin{multline}\label{prelimlowbound}
\log|L(1,f)|\ge\tRe\sum_{n\le  x}\frac{a_f(n)\Lambda(n)}{n\log n}\frac{\log(\tfrac{x}n)}{\log x}-\frac{\log C(f)}{(\sqrt{x}-1)\log x}\left(1+\frac{1+1/\sqrt{x}}{\log x}\right)\\
+\left(\frac{1}{\log x}+\frac{2}{\sqrt{x}\log^2x}\right)\left(1-\frac1{\sqrt{x}}\right)^{-2}\tRe\sum_{n\le x}\frac{a_f(n)\Lambda(n)}{n}\left(1-\frac{n}x\right)-\frac{2.05d}{(\sqrt{x}-1)^2}.
\end{multline}
To continue, we see from Lemma \ref{LamLiSound2.4} if $x\ge 132$ we have 
\begin{equation*}
\tRe\sum_{n\le x}\frac{a_f(n)\Lambda(n)}{n}\left(1-\frac{n}x\right)\ge d(1-\log x),
\end{equation*}
 thus as in \cite[pg 18  line 11  ]{LamLiSound} we have
 \begin{equation*}
\left(\left(1-\frac1{\sqrt{x}}\right)^{-2} \left(\frac{1}{\log x}+\frac{2}{\sqrt{x}\log^2x}\right)-\frac1{\log x}\right)\tRe\sum_{n\le x}\frac{a_f(n)\Lambda(n)}{n}\left(1-\frac{n}x\right)\ge -\frac{2d}{\sqrt{x}}.
 \end{equation*}

Using this in \eqref{prelimlowbound}  we have 
\begin{multline*}
\log |L(1,f)|\ge \tRe\sum_{n\le  x}a_f(n)\Lambda(n)\left(\frac1{n\log n}-\frac1{x\log x}\right)-\frac{\log C(f)}{(\sqrt{x}-1)\log x}\left(1+\frac{1+1/\sqrt{x}}{\log x}\right) -\frac{2d}{\sqrt{x}}-\frac{2.05d}{(\sqrt{x}-1)^2}\\
\ge \tRe\sum_{n\le  x}a_f(n)\Lambda(n)\left(\frac1{n\log n}-\frac1{x\log x}\right)-\frac{\log C(f)}{(\sqrt{x}-1)\log x}\left(1+\frac{1+1/\sqrt{x}}{\log x}\right) -\frac{9d}{4\sqrt{x}}.
\end{multline*}
%
 We apply Lemma \ref{generalizedLLS5.1}, thus guaranteeing that the first term in the right hand side is smallest when $a_f(p)=-d$ for every prime $p\le x$. Therefore, we have 
\begin{multline*}
\log |L(1,f)|\ge d\sum_{p^k\le  x}\Lambda(p^k)(-1)^k\left(\frac1{p^k\log p^k}-\frac1{x\log x}\right)-\frac{\log C(f)}{(\sqrt{x}-1)\log x}\left(1+\frac{1+1/\sqrt{x}}{\log x}\right) -\frac{9d}{4\sqrt{x}}.
\end{multline*}
Following the discussion after \cite[Equation 5.3]{LamLiSound} we see that 
\begin{multline*}
\log |L(1,f)|\ge d\left(-\log\log x-\gamma+\log\zeta(2)+\frac1{\log x}-\frac{8}{5\sqrt{x}}\right)-\frac{\log C(f)}{(\sqrt{x}-1)\log x}\left(1+\frac{1+1/\sqrt{x}}{\log x}\right)-\frac{9d}{4\sqrt{x}}.
\end{multline*}
Our choice of $x=\frac{\log^2C(f)}{4d^2}$ gives $-\log C(f)\ge -2d\sqrt{x}-1$ so that with a little calculation one obtains 
\begin{equation*}
\log |L(1,f)|\ge d\left(-\log\log x-\gamma+\log\zeta(2)-\frac1{\log x}-\frac2{\log^2x}-\frac{9}{2\sqrt{x}}\right).
\end{equation*}
Exponentiating both sides gives
\begin{align*}
\frac1{|L(1,f)|}&\le \left(e^{\g}\frac{6}{\pi^2}\right)^d\log^dx\exp\left(\frac{d}{\log x}+\frac{2d}{\log^2x}+\frac{9d}{2\sqrt{x}}\right)\\
&\le \left(e^{\g}\frac{6}{\pi^2}\right)^d\log^dx\left(1+\frac{d}{\log x}+\frac{dJ_1(d)}{\log^2x}+\frac{dJ_2(d)}{2\sqrt{x}}\right), 
\end{align*}
where
 \[J_1(d)=2+\left(1+\frac4{\log x}+\frac4{\log^2x}\right)\sum_{k=0}^{\infty}\frac{d^{k+1}}{(k+2)!}\left(\frac{1}{\log x}+\frac{2}{\log^2x}+\frac{9}{2\sqrt{x}}\right)^k,\] 
 and
  \[J_2(d)=9+\left(\frac{18}{\log x}+\frac{18}{\log^2 x}+\frac{81}{2\sqrt{x}}\right)\sum_{k=0}^{\infty}\frac{d^{k+1}}{(k+2)!}\left(\frac{1}{\log x}+\frac{2}{\log^2x}+\frac{9}{2\sqrt{x}}\right)^k.\]
  Replacing $x$ we get 
\begin{align*}
 \frac1{|L(1,f)|}&\le \left(2e^{\g}\frac{6}{\pi^2}\right)^d\left((\log\log C(f)-\log2d)^d+\frac{d}2(\log\log C(f)-\log2d)^{d-1}\right.\\
 &\left.+\frac{dJ_1(d)}4(\log\log C(f)-\log2d)^{d-2}+\frac{d^2J_2(d)(\log\log C(f)-\log2d)^d}{\log C(f)}\right)
\end{align*}
Thus the theorem is proven.
\section{Acknowledgements}
The author would like to thank Dr. Lamzouri for his guidance on this project.

\end{document}